\documentclass[12pt]{amsart}
\usepackage[left=3cm,right=3cm,top=4cm,bottom=4cm]{geometry}
\usepackage{amsmath}
\usepackage{amsfonts}
\usepackage{amssymb}
\usepackage{blkarray}
\usepackage{lmodern}
\usepackage{blkarray, bigstrut}
\usepackage{booktabs}
\usepackage{mathrsfs}
\usepackage{xfrac} 
\usepackage[colorlinks=true]{hyperref}
\usepackage{bm}
\usepackage{lipsum}

\newcommand\blfootnote[1]{%
  \begingroup
  \renewcommand\thefootnote{}\footnote{#1}%
  \addtocounter{footnote}{-1}%
  \endgroup
}

\newtheorem{theorem}{Theorem}[section]
\newtheorem{lemma}[theorem]{Lemma}
\newtheorem{corollary}[theorem]{Corollary}
\newtheorem{conjecture}[theorem]{Conjecture}

\theoremstyle{definition}
\newtheorem{definition}[theorem]{Definition}
\newtheorem{example}[theorem]{Example}
\theoremstyle{remark}
\newtheorem{remark}[theorem]{Remark}
\theoremstyle{remark}
{
\newtheorem*{notation}{Notation}
}
\numberwithin{equation}{section}

\newcommand{\CC}{{\mathbb C}}
\newcommand{\QQ}{{\mathbb Q}}
\newcommand{\ZZ}{{\mathbb Z}}
\newcommand{\NN}{{\mathbb N}}
\newcommand{\PP}{{\mathbb P}}
\newcommand{\bigslant}[2]{{\raisebox{.2em}{$#1$}\left/\raisebox{-.2em}{$#2$}\right.}}
\DeclareRobustCommand{\rchi}{{\mathpalette\irchi\relax}}
\newcommand{\irchi}[2]{\raisebox{\depth}{$#1\chi$}} % inner command, used by \rchi
\hypersetup{
  pdftitle={},
  pdfauthor={},
  pdfsubject={},
  pdfkeywords={},
  colorlinks=true,
  breaklinks=true,
  bookmarksopen=true,
  bookmarksnumbered=true,
  pdfpagemode=UseOutlines,
  plainpages=false,
  linkcolor=black,
  citecolor=black,
  anchorcolor=black,
  urlcolor  = black
  }
\begin{document}

\title{A Cluster Structure on the Coordinate Ring of Partial Flag Varieties}

\author{Fayadh Kadhem}
\address{Mathematics Department\\
Louisiana State University\\
Baton Rouge, Louisiana}
\email{fkadhe1@lsu.edu}

\date{\today}
\maketitle

\begin{abstract}
The main goal of this paper is to show that the (multi-homogeneous) coordinate ring of a partial flag variety $\CC[G / P_K^{-}]$ contains a cluster algebra if $G$ is any semisimple complex algebraic group. We use derivation properties and a special lifting map to prove that the cluster algebra structure $\mathcal{A}$ of the coordinate ring $\CC[N_K]$ of a Schubert cell constructed by Goodearl and Yakimov can be lifted, in an explicit way, to a cluster structure $\widehat{\mathcal{A}}$ living in the coordinate ring of the corresponding partial flag variety. Then we use a minimality condition to prove that the cluster algebra $\widehat{\mathcal{A}}$ is equal to $\CC[G / P_K^{-}]$ after localizing some special minors.
\end{abstract}

\maketitle

\section{Introduction}
\label{introduction}
\blfootnote{The research of the author has been supported by NSF grant DMS-2131243.}
Cluster algebras were introduced in 2002 by Fomin and Zelevinsky and they have rapidly become one of the active areas in mathematics. This is due to their deep relations to other areas of of mathematics like representation theory, combinatorics, homological algebra, algebraic geometry, Poisson geometry, Teichm{\"{u}}ller theory and mathematical physics. On the other hand, the study of partial flag varieties is significant in representation theory and algebraic geometry. The first connection between these two studies appeared in Scott's work on Grassmannians and cluster algebras ~\cite{S} in 2003. In 2008, Gei{\ss}, Leclerc and Schr{\"{o}}er ~\cite{GLS} showed that, in some simply-laced cases, namely $A_n$ and $D_4$, the localization of the (multi-homogeneous) coordinate ring of a partial flag variety by non-minuscule minors matches the localization of some cluster structure by the same minors. They conjectured that this is true in the general, that is, when the type of $G$ is arbitrary. This paper proves this conjecture with respect another localization. The main ideas of the proof of ~\cite{GLS} motivate our work here.
Indeed, Gei{\ss}, Leclerc and Schr{\"{o}}er proved that the coordinate ring of a partial flag variety contains a cluster structure by showing the following:
\begin{enumerate}
    \item The coordinate ring of a Schubert cell has a cluster algebra structure $\mathcal{A}$.
    \item The cluster algebra $\mathcal{A}$ of the previous step can be lifted to some special cluster algebra $\widehat{\mathcal{A}}$ that lives in the coordinate ring of the partial flag variety corresponding to the coordinate ring of the cell of the previous step.
    \item The cluster algebra $\widehat{\mathcal{A}}$ coincides with the coordinate ring of the partial flag variety after localization with respect to some special minors.
\end{enumerate}
Although the first step was only conjectured in \cite{GLS}, it was fully proved in \cite{GLS3}. Moreover, despite the fact that we prove the second step independently, it was also generalized to the non-simply-laced ones by Demonet in \cite{D}.\\

Unfortunately, some essential tools of the proof of Gei{\ss}, Leclerc and Schr{\"{o}}er were based on the fact that they work on the simply-laced case. In fact, they used some categorification in their work, which works in the simply-laced case only, to show the first and the second steps, while they treated the third step for types $A_n$ and $D_4$ case by case. Because of that, the generalization we seek has to use some other results.\\

Goodearl and Yakimov ~\cite{GY, GY3} proved that the coordinate ring of any Schubert cell admits a cluster structure. Moreover, their construction matches the one of ~\cite{GLS} in the simply-laced case, yet it gives an explicit cluster structure to the coordinate ring of a cell in the non-simply-laced as well. It is worth mentioning here that in spite of the fact that the theory of cluster algebras is a mix between combinatorics and algebra, the work of Goodearl and Yakimov was almost purely algebraic.\\

The work of ~\cite{GY} and ~\cite{GY3} enables us to go back to the strategy of ~\cite{GLS}, that is, the three steps mentioned above, and follow them to prove that the coordinate ring of a partial flag variety contains a cluster algebra, no matter if we are in the simply-laced or the non-simply-laced case. Of course, we have to find different ways to treat steps 2 and 3, but thanks to Goodearl and Yakimov, the first step is already there.\\

\noindent In particular, to get steps 2, we proved the following theorem:
\begin{theorem}
Let $\left\{ (\textnormal{\textbf{x}},{B}) \right\}$ be the collection of seeds of the cluster algebra $\mathcal{A}$ of $\CC[N_K]$. The corresponding collection of pairs $\left\{ (\widehat{\textnormal{\textbf{x}}},\widehat{B}) \right\}$ constructed in Definition \ref{main definition} forms a collection of seeds related by mutation. In other words, if $(\textnormal{\textbf{x}},{B})$ and $(\textnormal{\textbf{x}}',{B'})$ are two seeds of the coordinate ring of the cell $\CC[N_K]$ such that $(\textnormal{\textbf{x}}',{B'}) = \mu_k(\textnormal{\textbf{x}},{B})$, then correspondingly $(\widehat{\textnormal{\textbf{x}}'},\widehat{B'})= \mu_k (\widehat{\textnormal{\textbf{x}}},\widehat{B})$. In particular, if $ (\textnormal{\textbf{x}}_0,{B}_0)$ is an initial seed of $\mathcal{A}=\CC[N_K]$ then $(\widehat{\textnormal{\textbf{x}}_0},\widehat{B_0})$ is an initial seed of a cluster algebra $\widehat{\mathcal{A}} \subset \CC[G/P_K^-]$.
\end{theorem}

\noindent For step (3), we actually proved that:
\begin{theorem}
%The homogeneous coordinate ring of the flag variety $\CC[G/P_K^-]$ equals the cluster algebra $\widehat{\mathcal{A}}$.
The localization of the homogeneous coordinate ring of the flag variety $\CC[G/P_K^-]$ by $\Delta_{\varpi_j, \varpi_j}$, where $j \in J$, equals the localization of the cluster algebra $\widehat{\mathcal{A}}$ by the same elements. In symbols,
$$\CC[G/P_K^-][\Delta_{\varpi_j, \varpi_j}^{-1}]_{j\in J}=\widehat{\mathcal{A}}[\Delta_{\varpi_j, \varpi_j}^{-1}]_{j\in J}.$$
\end{theorem}

\noindent Basically, we complete the second step of the strategy of Gei{\ss}, Leclerc and Schr{\"{o}}er in the first theorem and then do the third step in the second theorem. It is worth to mention here that the localization ~\cite{GLS} is over a the minors that are indexed by the set $J$ and are not minuscule, while we localize by the minors that are indexed by $J$ and omit the second condition.

Here is an outline of how the paper is organized: In the following section, we give the reader an overview of the structure of cluster algebras, while in Section 3 we go through the needed results from partial flag varieties. However, in Section 4, we focus on the highlights of the work of Goodearl and Yakimov. Indeed, we discuss the relation between Poisson geometry and cluster algebras and show how the cluster algebra $\mathcal{A}$ of the coordinate of a Schubert cell looks based on the structure of Goodearl and Yakimov. In fact, it is shown in Theorem \ref{Schubert cell} that the variables of their initial extended cluster are nothing but restrictions of some special homogeneous elements of the corresponding coordinate ring of a partial flag variety, called \textit{generalized minors}. Also, the exchange matrix of their work is given explicitly in the same theorem. Using the intuition from the work of ~\cite{GLS}, we then assigned, in Definition \ref{main definition}, a pair $(\widehat{\textnormal{\textbf{x}}},\widehat{B})$ to each seed $(\textnormal{\textbf{x}},B)$ of $\mathcal{A}$. In this new pair, $\widehat{\textnormal{\textbf{x}}}$ consists of the lifting of the same elements of $\textnormal{\textbf{x}}$ plus the generalized minors $\Delta_{\varpi_j,\varpi_j}$ for which the restriction is 1 in the coordinate ring of the cell. Also, the matrix $\widehat{B}$ is the matrix $B$ together with some additional rows given in some special form. After that, we show in Theorem \ref{main theorem} that these pairs are actually seeds of some cluster algebra $\widehat{\mathcal{A}}$ sitting inside the coordinate ring of the partial flag variety. Moreover, two pairs are related by a mutation if their corresponding original seeds of $\mathcal{A}$ are. This finishes step 2 of the strategy of ~\cite{GLS}. Subsequently, we use a minimality property in Theorem \ref{equality theorem} to show that the cluster algebra $\widehat{\mathcal{A}}$ is indeed equal to the coordinate ring of the partial flag variety.\\

In fact, it is an important problem to understand the relationship between the cluster structures of Demonet \cite{D} and ours. We plan to return to this in a future publication.
%The third section of this paper introduces the idea of graded cluster algebras. Although it is not used in the proof of the results of this paper, we strongly believe that it can be used to formulate a valid proof of the same main theorem. This is because the coordinate ring of a cell has two equivalent definitions. The first, which is what we use throughout the paper, is that the coordinate ring of a cell is the quotient of the coordinate ring of the corresponding partial flag variety by the ideal generated by $(\Delta_{\varpi_j, \varpi_j}-1)$ where $j$ runs through a subset $J$ of the set $I$ of the Dynkin diagram vertices. The second defines it as the subring of degree 0 of the localization of the coordinate ring of the partial flag variety by the same elements $\Delta_{\varpi_j, \varpi_j}$ in which $j \in J$.
%%%%%%%%%%%%%%%%%%%%%%%%New Section
%%%%%%%%%%%%%%%%%%%%%%%%Cluster Algebras
\section{Cluster Algebras}
This section gives an overview of the construction of cluster algebras and the main concepts. For more details about this, the reader is referred to  ~\cite{FWZ}, ~\cite{FZ}, or ~\cite{GSV}.
\begin{definition}
In our setting, the term \textit{ambient field} will be referring to a field $\mathcal{F}$ that is isomorphic to $\CC(x_1,...,x_n,...,x_m)$, where $\{x_1,...,x_n,...,x_m \}$ is an algebraically independent generating set.
\end{definition}

\begin{remark}
We usually write $\CC(x_1,...,x_n,...,x_m)$ instead of  writing $\CC(x_1,...,x_m)$ to emphasize that there is a distinction between the first $n$-variables and the rest $(m-n)$-ones. This distinction will become clear in the following sequence of definitions and remarks.
\end{remark}

\begin{definition}
A (\textit{labeled}) \textit{seed} is a pair $( \widetilde{{\textbf{x}}},\widetilde{B})$ where $\widetilde{{\textbf{x}}}$ is a tuple of algebraically independent variables $\widetilde{{\textbf{x}}}=(x_1,...,x_n,...,x_m)$ generating an ambient field $\mathcal{F}$ and $\widetilde{B}$ is an $m \times n$ matrix whose northwestern $n \times n$ submatrix $B$ is \textit{skew-symmetrizable}, that is, can be transformed to a skew-symmetric matrix by multiplying each row $r_i$ by some nonzero integer $d_i$. The tuple $\widetilde{{\textbf{x}}}$ is called an \textit{extended cluster}, where its first $n$-variables are called the \textit{cluster} (or \textit{mutable}) variables and the next $(m-n)$-variables are called the \textit{coefficient} (or \textit{frozen}) variables. The tuple $\textbf{x}=(x_1,...,x_n)$ is called a \textit{cluster}. In the same context, the northwestern $n \times n$ submatrix $B$ of $\widetilde{B}$ is called the \textit{exchange} matrix, while the matrix $\widetilde{B}$ is called the \textit{extended exchange} matrix.
\end{definition}

\begin{remark}
Sometimes the skew-symmetrizable matrix is replaced by a \textit{quiver} $Q$, which is a directed graph with $n$-\textit{mutable} and $(m-n)$-\textit{frozen} vertices such that it has no loops, no 2-oriented-cycles and no edges between two frozen vertices. In fact, each quiver gives rise to an $m \times n$ skew-symmetrizable matrix $\widetilde{B}(Q)$, where its entries are given by
$$b_{ij}=\begin {cases}
\# (i\rightarrow j), & \text{if}\ i > j,\\ 
0, & \text{if}\ i=j,\\ 
- \# (i \leftarrow j), & \text{if}\ i < j;
\end{cases}$$
where $\#(i \rightarrow j)$ is the number of arrows from $i$ to $j$ and $\#(i \leftarrow j)$ is the number of arrows from $j$ to $i$.
\end{remark}

\begin{definition}
Let $(\widetilde{{\textbf{x}}},\widetilde{B})$ be a seed. A \textit{mutation} $\mu_k$ at $k\in [1,n]$ is a transformation to a new seed $\mu_k(\widetilde{{\textbf{x}}},\widetilde{B})=(\widetilde{{\textbf{x}}}',\widetilde{B}')$, where the entries of the matrix $\widetilde{B}'$ are given by

\begin{equation}
    \label{eq1}
b'_{ij}=\begin {cases}
-b_{ij}, & \text{if}\ i=k \text{ or } j=k,\\ 
b_{ij}+\dfrac{|b_{ik}|b_{kj} + b_{ik}|b_{kj}|}{2}, & \text{otherwise};\\
\end{cases}
\end{equation}
and $\widetilde{{\textbf{x}}}'=(x'_1,...,x'_m)$, where $x'_i=x_i$ if $i\neq k$ and
$$x_k x'_k= \prod_{b_{ik}>0} x_i^{b_{ik}}+\prod_{b_{ik}<0} x_i^{-b_{ik}}.$$
Two seeds are said to be \textit{mutation equivalent} if one of them can be obtained from the other one by a sequence of mutations.
\end{definition}

\begin{remark}
It is not hard to verify that $\mu_k$ is an \textit{involution}, that is,
$$\mu_k(\mu_k(\widetilde{{\textbf{x}}},\widetilde{B}))=(\widetilde{{\textbf{x}}},\widetilde{B}).$$
\end{remark}

\begin{remark}
Let us start with an \textit{initial seed} $(\widetilde{{\textbf{x}}},\widetilde{B})$. It is known that any mutable variable can be obtained from $(\widetilde{{\textbf{x}}},\widetilde{B})$ by some sequence of mutations at some mutable indices. Therefore, knowing an initial seed gives a full picture of the mutable variables, and thus all of the extended clusters.
\end{remark}

\begin{definition}
Let $(\widetilde{{\textbf{x}}},\widetilde{B})$ be a seed. Let $\rchi$ be the set of all possible mutable variables, that is, the mutable ones of the initial seed or the mutable ones generated by any sequence of mutations applied on the initial seed. Let $\mathcal{R}$ be the polynomial ring $\mathcal{R}=\CC[x_{n+1},...,x_m]$, where $x_{n+1},...,x_m$ are the frozen variables of the seed $(\widetilde{{\textbf{x}}},\widetilde{B})$. The \textit{cluster algebra} (of \textit{geometric type}) is the algebra $\mathcal{A}=\mathcal{R}[\rchi]$, the polynomial algebra of all variables (mutable or frozen).
\end{definition}

\begin{remark}
Since an initial seed $(\widetilde{{\textbf{x}}},\widetilde{B})$ provides full information about its corresponding cluster algebra, we shall denote the latter by $\mathcal{A}(\widetilde{{\textbf{x}}},\widetilde{B})$.
\end{remark}

\begin{definition}
Let $(\widetilde{{\textbf{x}}},\widetilde{B})$ be a seed. The \textit{rank} of the seed or its corresponding cluster algebra is the number of its mutable variables, while the number of all variables of the seed is referred to as the \textit{cardinality} of the seed. Thus in our setting above, the rank of $(\widetilde{{\textbf{x}}},\widetilde{B})$ is $n$ and the cardinality of it is $m$.
\end{definition}

\begin{definition}
A cluster algebra $\mathcal{A}(\widetilde{{\textbf{x}}},\widetilde{B})$ is said to be \textit{of finite type} if it has a finite number of mutable variables. Otherwise it is said to be \textit{of infinite type.}
\end{definition}

\section{Partial Flag Varieties}
This section captures the required overview from the partial flag varieties. We need first review the definition of a partial flag variety and look at some facts about its coordinate ring. Other useful overviews, with probably more details about this, can be found in ~\cite{GLS}, ~\cite{GSV}, or ~\cite{J}.

\begin{remark}
It is known that each semisimple group induces a \textit{Cartan matrix} whose information can be encoded in the corresponding \textit{Dynkin diagram}. One of the significant consequences of this is that every semisimple Lie algebra is fully characterized, up to isomorphism, by its Dynkin diagram.
%The figure below shows the list of Dynkin diagrams, where the subscript denotes the number of vertices.
\end{remark}

%\tble{A/{},B/{},C/{},D/{},E/6,E/7,E/8,F/4,G/2}

\begin{remark}
From now on, the set $I$ denotes the vertex set of the Dynkin diagram $\Delta$ corresponding to $G$.
\end{remark}

\begin{definition}
A \textit{parabolic} subgroup $P$ of $G$ is a closed subgroup that lies between $G$ and some Borel subgroup $B$.
\end{definition}

\begin{example}
\begin{enumerate}
    \item Any Borel subgroup $B$ is parabolic.
    %\item The opposite Borel subgroup $B^-$ of any Borel subgroup $B$ is also parabolic.
    \item Fix a nonempty subset $J \subset I$ and let $K=I \setminus J$. Denote by $x_i(t)$ $(i\in I, t \in \CC)$ the simple root subgroups of the unipotent radical $N$ of $B$ and denote by $y_i(t)$ the simple root subgroups of the unipotent radical $N^-$ of $B^-$. The subgroup $P_K$ generated by $B$ and the one-parameter subgroups $y_k(t)$ $(k\in K, t\in \CC)$ is parabolic. Similarly, the subgroup $P_K^-$ generated by $B^-$ and the one-parameter subgroups $x_k(t)$ $(k \in K, t\in \CC)$ is a parabolic subgroup.
\end{enumerate}
\end{example}

\begin{definition}
A quotient $G/P$ is called a (\textit{partial}) \textit{flag variety} if $P$ is a parabolic subgroup of $G$.
\end{definition}

\begin{remark}
It is known that any parabolic subgroup is conjugate to a parabolic subgroup of the form $P_K$. This somehow, in many cases, reduces the study of partial flag varieties to the ones of the form $G/{P_K}$.
\end{remark}

\begin{remark}
The partial flag variety $G / P_K^-$ can be naturally embedded as a closed subset of the product of projective spaces
$$\prod_{j \in J} \PP(L(\varpi_j)^*),$$
where $\varpi_j$ is a fundamental weight of $G$, and for a dominant weight $\lambda$, the corresponding $L(\lambda)$ is the finite-dimensional irreducible $G$-module with highest weight $\lambda$; and $L(\lambda)^*$ denotes the right $G$-module obtained by twisting the action of $G$. As a terminology, the $L(\varpi_i)$'s are called the \textit{fundamental representations.}
\end{remark}

\begin{remark}
Let $\Pi_J \cong \NN^J$ denote the monoid of dominant integral weights of the form $\lambda = \sum_{j \in J} a_j \varpi_j$, where $a_j \in \NN$. The multi-homogeneous coordinate ring $\CC[G/P_K^-]
$ is a $\Pi_J$-graded algebra. In particular,
$$\CC[G/P_K^-]= \bigoplus_{\lambda \in \Pi_J} L(\lambda).$$
One of the significant results is that $\CC[G/P_K^-]$ can be identified with the subalgebra of $\CC[G/{N^-}]$ generated by the homogeneous elements of degree $\varpi_j$, where $j \in J$.
\end{remark}

\begin{remark}
For a Weyl group $W$ of $G$, the longest element in this paper will always be denoted by $w_0$ and the Coxetor generators will be denoted by $s_i$ where $i$ runs in $I$.\\
The notation of the length of some $w \in W$ will be $\ell (w).$ The Chevalley generators of the Lie algebra $\mathfrak{g}$ of $G$ are denoted $e_i,f_i,h_i$, where again $i$ runs in $I$. The $e_i$'s here generate $\textnormal{Lie}(N)= \mathfrak{n}.$ An important consequence of this is that $N$ acts naturally from the left and right on $\CC[N]$ by the following left and right actions respectively:
$$(x \cdot f)(n) = f(nx), \quad  (f \in \CC[N] \textnormal{ and } x,n \in N),$$
$$(f \cdot x)(n) = f(xn), \quad  (f \in \CC[N] \textnormal{ and } x,n \in N).$$
One might differentiate these two actions to get left and right actions of $\mathfrak{n}$ on $\CC[N]$, respectively.
\end{remark}

\begin{notation}
The right action of $e_i$ on $f \in \CC[N]$ will be denoted by $e ^ \dagger _i (f) := f \cdot e_i.$
\end{notation}

\begin{remark}
For each simple reflection $s_i \in W$, let $\overline{s_i}:=\exp(f_i) \exp(e_i) \exp (f_i)$. If $w=s_{i_1}...s_{i_r}$ with $r$ being the length of $w$, then define $\overline{w}=\overline{s_{i_1}}...\overline{s_{i_r}}$. Let $G_0=N^-HN$ be the open set of $G$ consisting of elements having Gaussian decomposition. Indeed, each $x\in G_0$ can be uniquely represented as
$$x=[x]_{-}[x]_0[x]_+,$$
where $[x]_{-}\in N^{-},$ $[x]_0\in H,$ $[x]_+ \in N$. Let $V_i^+$ be the irreducible representation whose highest weight is $\varpi_i$ and highest weight vector is $v_i^+$. For any $h\in H$ one has that $v_i^+$ is an eigenvector, that is, $hv_i^+=[h]^{\varpi_i} v_i^+$ and $[h]^{\varpi_i} \in \CC \setminus \{0\}$. This gives the following definition introduced by Fomin and Zelevinsky in \cite{FZ0}.
\end{remark}

\begin{definition}
For $u,v \in W$ and $i \in I$ define the \textit{generalized minor} to be the regular function on $G$ given by
$$\Delta_{u\varpi_i, v\varpi_i}(x)=[\overline{u}^{-1}x \overline{v}]_0^{\varpi_i}.$$
\end{definition}

\begin{remark}
  The distinguished elements $\Delta_{\varpi_j, w(\varpi_j)}$, $(w\in W)$, are of degree $\varpi_j$ (see 2.3 in \cite{BFZ} or section 2 and 6 in \cite{GLS} for more details). They make the coordinate ring of the cell and the coordinate ring of the corresponding flag variety related by the following:
  $$\CC[N_{K}]=\bigslant{{\CC[G/P_{K}^{-}]}}{(\Delta_{\varpi_j,\varpi_j}-1)}_{j \in J}.$$
  The generalized minors are nothing but a generalization of the flag minors of $SL_n$. Their significance in the cluster structure of the coordinate ring of partial flag varieties will be seen in section \ref{main section}.
 \end{remark}

%NEW SECTION
\section{Preleminaries from Poisson Algebras}
In ~\cite{GY} and ~\cite{GY3}, Goodearl and Yakimov made the relationship between the coordinate ring of Schubert cells and cluster algebras clear and explicit. They proved that each such coordinate ring admits a cluster structure. Thus, since the coordinate ring of any cell is the quotient of the coordinate ring of some flag variety modded out by some generalized flag minors, it is so obvious that the result of Goodearl and Yakimov can play an important role in this paper. Their results were based on Poisson geometry and so we capture here the main elements that we need from their work. More details about the relation between Poisson geometry and cluster algebras can be found in ~\cite{GSV} and ~\cite{GY}.
\begin{definition}
\begin{enumerate}
    \item A \textit{Poisson bracket} $\{-,-\}$ is a Lie bracket that is a derivation also in each variable for the associative products.
    \item A \textit{Poisson algebra} is a commutative algebra $R$ together with a Poisson bracket.
    \item For $a\in R$ the \textit{Hamiltonian associated} with $a$ is the derivation $\{a,-\}$.
    \item A \textit{Poisson ideal} of $R$ is an ideal $I$ such that $\{R,I\} \subset I$.
\end{enumerate}
\end{definition}

\begin{remark}
The Poisson bracket of a Poisson algebra $R$ induces a Poisson bracket on any quotient of $R$ by a Poisson ideal.
\end{remark}

\begin{definition}
Define the \textit{Poisson-Ore extensions} to be $B[x;\sigma,\delta]_p$ where $B$ is a Poisson algebra, $B[x;\sigma,\delta]_p=B[x]$ is a polynomial ring and $\sigma , \delta$ are suitable Poisson derivations on $B$ such that for any $b \in B$ we have
$$\{x,b\}= \sigma(b)+\delta(x).$$
For an iterated Poisson-Ore extension $$R=\mathbb{K}[x_1]_p[x_2;\sigma_2,\delta_2]_p \cdots [x_m;\sigma_m,\delta_m]_p$$
and $k\in [0,m]$, define
$$R_k=\mathbb{K}[x_1,...,x_k]=\mathbb{K}[x_1]_p[x_2;\sigma_2,\delta_2]_p \cdots [x_k;\sigma_k,\delta_k]_p,$$
where $R_0= \mathbb{K}.$
\end{definition}

\begin{definition}
A \textit{Poisson-CGL extension} is an iterated Poisson-Ore extension $R$ as above that is endowed with a rational Poisson action of a torus $\mathcal{H}$ such that
\begin{enumerate}
    \item The elements $x_1,...,x_k$ are $\mathcal{H}$-eigenvectors;
    \item The map $\delta_k$ is locally nilpotent on $R_{k-1}$ for any $k\in [2,m]$;
    \item For any $k\in [1,m]$ there is an $h_k \in \textnormal{Lie}\mathcal{H}$ such that $\sigma_k=h_k |_{R_{k-1}}$ and the $h_k$-eigenvalue of $x_k$ nonzero and denoted by $\lambda_k.$
\end{enumerate}
\end{definition}

\begin{definition}
Let $R$ be a Noetherian Poisson domain. An element $p \in R$ is called a \textit{Poisson-prime element} if any of the following equivalent conditions hold:
\begin{enumerate}
    \item The ideal $(p)$ is a prime ideal and it is a Poisson ideal.
    \item The element $p$ is a prime element of $R$ such that $p | \{p,-\}.$
    \item $[$In the case $\mathbb{K}=\CC]$: The element $p$ is a prime element of $R$ and the zero locus $V(p)$ is a union of symplectic leaves of the maximal spectrum of $R$.
\end{enumerate}
\end{definition}

One of the great successes is due to the work of Goodearl, Yakimov when they proved the following:

\begin{theorem}
Every symmetric Poisson-CGL extension $R$ such that $\lambda_l / \lambda_j \in \QQ_{>0}$ for all $l,j$ has a canonical cluster algebra structure that coincides with its upper cluster algebra. 
\end{theorem}

\begin{remark}
The cluster variables in the constructions of Goodearl, Yakoimov are the unique homogeneous Poisson-prime elements of Poisson-CGL (sub)extensions not belonging to smaller subextensions. The mutation matrices of their seeds can be computed using linear systems of equations that come from the Poisson structure.
\end{remark}

A significant consequence of the work of Goodearl, Yakimov is:
%\begin{theorem}
%Let $V$ be a Lie-theoretic variety corresponding to the field $\mathbb{K}$. The coordinate ring $\mathbb{K}[V]$ has a canonical cluster algebra structure $\mathbb{K}[V]=\mathcal{A}(\widetilde{x},\widetilde{B})$.
%\end{theorem}

\begin{theorem} \label{Schubert cell}
The coordinate ring $\CC[N_K]$ has a canonical cluster algebra structure.
\end{theorem}
\begin{proof}
Throughout this proof, the notation $e_k$ means the $k$th vector of the standard basis of $\ZZ^m$, the notation $a[j,k]$ is given by
$$a[j,k]:= \|(w_{[j,k]}-1)\varpi_{i_k} \|^2/4 \in \dfrac{1}{2}\ZZ ,$$
and the notation $S(w)$ is the \textit{support of $w$} and is given by
$$S(w):=\{i \in I \mid s_i \leq w \}=\{i \in I \mid i=i_k \textnormal{ for some } k \in [1,m] \}.$$
Also, set
$$p(k):=\begin{cases}
\textnormal{max}\{j < k \textnormal{ }| \textnormal{ } i_j=i_k \},& \text{if such } j \text{ exists;}\\
- \infty, & \text{otherwise.}
\end{cases}$$
$$s(k):=\begin{cases}
\textnormal{min}\{j > k \textnormal{ }| \textnormal{ } i_j=i_k \},& \text{if such } j \text{ exists;}\\
\infty, & \text{otherwise.}
\end{cases}$$
From Theorem 7.3 in \cite{GY3} we know that the quantum Schubert cell, denoted by $A_q(\mathfrak{n}_{+}(w))_{\mathcal{A}^{1/2}}$, has the quantum cluster structure given by the equation
$$A_q(\mathfrak{n}_{+}(w))_{\mathcal{A}^{1/2}}=\bm{\mathsf{A}}(M^w,\widetilde{B}^w,\varnothing)_{\mathcal{A}^{1/2}}=\bm{\mathsf{U}}(M^w,\widetilde{B}^w,\varnothing)_{\mathscr{A}^{1/2}},$$
where the extended cluster variables are given by
$$M^w(e_j)=q^{a[1,j]}D_{\varpi_{i_j},w_{\leq j} \varpi_{i_j}},$$ 
for all $j \in [1,m]$, in which
$$D_{\varpi_{j},w(\varpi_{j})}
=\textnormal{proj}({\Delta_{\varpi_{j},w(\varpi_{j})}}),$$
in which the frozen variables are the ones indexed by $j \in [1,m]$ such that $s(j)=\infty.$ The map
$$\textnormal{proj}:\CC[G/P_{K}^{-}] \to \CC[N_K]$$
denotes the standard projection from $\CC[G/P_{K}^{-}]$ to $\CC[N_K]$.
The exchange matrix $\widetilde{B}^w$ is of size $m \times (m-|S(w)|)$ and its $j \times k$ entry is given by
$$(\widetilde{B}^w)_{jk} = \begin{cases}
1, & \text{if } j=p(k), \\
-1, & \text{if } j=s(k),\\
a_{i_j i_k}, & \text{if } j<k<s(j)<s(k),\\
-a_{i_j i_k}, & \text{if } k<j<s(k)<s(j),\\
0, & \text{otherwise;}
\end{cases}$$
where the entry $a_{i_j i_k}$ is the same $i_j \times i_k$ entry of the Cartan matrix of the same type.
Thus, by corollary 3.7 in \cite{GLS2}, it follows that
$$\mathbb{C} \otimes A_q(\mathfrak{n}_{+}(w))_{\mathscr{A}^{1/2}} \cong A(\widetilde{B}^w).$$
On the other hand, by (4.7) in \cite{Y}, we know that the left hand side is isomorphic to the quotient of $A_q(\mathfrak{n}_{+}(w))_{\mathcal{A}^{1/2}}$ by $(q-1)$. Consequently, we get the desired cluster structure in the classical case whose exchange matrix is $\widetilde{B}^w$ and cluster variables are
$D_{\varpi_{i_k},w_{\leq k} \varpi_{i_k}}$.
\end{proof}

%NEW SECTION
\section{Cluster Algebra Structure on $\CC[G/P_K^-]$} \label{main section}
In the work of Gei{\ss}, Leclerc and Schr{\"{o}}er ~\cite{GLS}, they proved that $\CC[G/P_K^-]$ up to localization admits a cluster structure if $G$ is simply-laced of type $A_n$ or $D_4$. Their work motivates our construction here. The idea is to translate their work, which was in terms of categorification, to another language that works in the general case.

\begin{notation}
The cluster structure on $\CC[N_K]$ constructed by the work of Goodearl and Yakimov will be denoted by $\mathcal{A}_J$, where $J$ and $K$ are as defined before. We may write $\mathcal{A}$ instead of $\mathcal{A}_J$ if the context is clear.
\end{notation}

\begin{lemma}\label{raise}
For every $f \in \CC [N_K]$ there exists a unique homogeneous element $\widetilde{f} \in \CC [G/P_{K}^{-}]$ such that its projection to $\CC[N_K]$ is $f$ and whose multi-degree is minimal with respect to the usual partial ordering obtained by the usual ordering of weights, that is, $\mu \preceq \lambda$ iff $\lambda - \mu$ is an $\NN$-linear combination of weights $\varpi_j$ $(j \in J)$. 
\end{lemma}
\begin{proof}
This is Lemma 2.4 in ~\cite{GLS}. Despite the fact that the main results of that paper are for the simply-laced case, this one is general and works for any type.
\end{proof}

\begin{remark}
The proof of the preceding lemma in ~\cite{GLS} involves the following important points:
\begin{enumerate}
    \item The notation $a_j(f)$ means the maximum of $\begin{Bmatrix} s \mid (e^ \dagger _j)^s f \neq 0\end{Bmatrix}$. In symbols
    $$a_j(f)= \max \begin{Bmatrix} s \mid (e^ \dagger _j)^s f \neq 0\end{Bmatrix}. $$
    \item The notation $\lambda (f)$ means
    $$ \lambda (f) = \sum_{j \in J} a_j(f) \varpi_j.$$
    \item The minimality in the previous lemma means that $\lambda (f)$ is minimal in the following sense: if $\tilde{\tilde{f}} \in L(\lambda)$ and $\textnormal{proj}(\tilde{\tilde{f}})=f$ then $\lambda (f) \preceq \lambda$. On the other hand, the projection of each piece $L(\lambda)$ to $\CC[N_K]$ is injective and so there if there is an element there whose projection is $f$, then it is unique in $L(\lambda)$. These two pieces of information together are the main ingredients in proving the existence and uniqueness of $\lambda(f).$
\end{enumerate}
\end{remark}

\begin{remark}
The endomorphisms $e ^ \dagger _j$ are derivations of $\CC[N_K]$. Thus, for all $f,g \in \CC[N_K]$ we have the following:
\begin{enumerate}
    \item The image of $fg$ under $e ^ \dagger _j$ is $$e ^ \dagger _j (fg) = e ^ \dagger _j (f) g + f e ^ \dagger _j (g);$$
    \item By Leibniz formula,
    $$(e ^ \dagger _j)^{a_j(f)+a_j(g)}(fg) = (e ^ \dagger _j)^{a_j (f)} (e ^ \dagger _j)^{a_j(g)}(g) \neq 0;$$
    \item For any integer $k \geq 1,$
    $$(e ^ \dagger _j)^{a_j(f)+a_j(g)+k}(fg) =0;$$
    \item Consequently,
    $$a_j (fg) = a_j(f) + a_j(g) .$$
\end{enumerate}
\end{remark}

\begin{lemma}\label{lifting}
For all $f,g \in \CC [N_K]$, we have $\widetilde{f \cdot g}=\widetilde{f} \cdot \widetilde{g}$. If for any $j \in J$, $a_j(f+g)=\max \{a_j(f),a_j(g) \}$, then there are some relatively prime monomials $\mu, \nu$ in the generalized minors $\Delta_{\varpi_j,\varpi_j}$ such that
$$\widetilde{f+g}=\mu \widetilde{f} + \nu \widetilde{g}.$$
\end{lemma}
\begin{proof}
Lemma 2.5 in ~\cite{GLS}.
\end{proof}
%\begin{proposition}
%If $f_1,...,f_l$ are elements of $\CC[N_K]$ in which their sum cannot be written as the sum of fewer terms, then $a_j(f_1+...+f_l)=\max \{a_j(f_1),...,a_j(f_l) \}$ for all $j\in J$.
%\end{proposition}
%\begin{proof}
%For any $j \in J$, since $e ^ \dagger _j$ is a derivation, it is clear that $(e ^ \dagger _j)^k (f_1)=e ^ \dagger _j(f_1)+...+e ^ \dagger _j(f_l)$
%
%Using the fact that
%$$\CC[N_K]=\left\{\dfrac{f}{\prod_{j \in J} \Delta_{\varpi_j,\varpi_j}^{a_j}} \mid f\in L \bigg(\sum_{j\in J} a_j \varpi_j \bigg) \right\}$$
%\end{proof}

\begin{remark} \label{cluster mutation}
Let $(\widetilde{{{\textnormal{\textbf{x}}}}},\widetilde{B})$ be a seed of the cluster algebra $\mathcal{A}= \CC[N_K]$. Then the mutation formula tells us that
$$x_k x_k' = M_k + L_k,$$
where $M_k,L_k$ are monomials in the variables $x_1,...,x_{k-1},x_{k+1},...,x_n.$ As a consequence of the previous lemma (c.f. ~\cite{GLS}) we get that
$$\widetilde{x_k} \widetilde{x'_k}=\mu_k\widetilde{M_k}+\nu_k \widetilde{L_k},$$
where $\mu_k$ and $\nu_k$ are relatively prime monomials in $\Delta_{{\varpi_{j}},{\varpi_{j}}}$ $(j \in J)$. This means that we can write $\mu_k$ and $\nu_k$ as
$$\mu_k = \prod_{j \in J} \Delta_{{\varpi_{j}},{\varpi_{j}}}^{\alpha_j} \quad \textnormal{ and } \quad \nu_k = \prod_{j \in J} \Delta_{ {\varpi_{j}}, {\varpi_{j}}}^{\beta_j}.$$
Consequently, it is reasonable to expect that the variables $\widetilde{x_i}$ form the cluster variables of some cluster algebra contained in $\CC[G/P_K^-]$. This was proved in type $A_n$ and $D_4$ by Gei{\ss}, Leclerc and Schr{\"{o}}er.
\end{remark}
%%%%%%%%%%%%%%%%%%%%%%%%%%%%%%%%%%%%%%%%%%
%%%%%%%%%%%%%%%%%%%%%%%%%%%%%%%%%%%%%%%%%%
%The example was here
\begin{definition}
A \textit{lift} of a cluster algebra $\mathscr{A}$ is a cluster algebra $\widetilde{\mathscr{A}}$ such that $\mathscr{A}$ is a quotient algebra of it. Alternatively, we may say that $\mathscr{A}$ \textit{can be lifted} to $\widetilde{\mathscr{A}}$.
\end{definition}

\begin{definition} \label{main definition}
For any seed $(\textnormal{\textbf{x}},{B})$ of the cluster algebra $\mathcal{A}_J=\CC [N_K]$ constructed in \cite{GY} define a new pair $(\widehat{\textnormal{\textbf{x}}},\widehat{B})$ of $\CC[G / P_{K}^-]$ by raising each variable $x$ of $(\textnormal{\textbf{x}},{B})$ to the variable $\widetilde{x}$ (see Lemma \ref{lifting}) preserving the same type (mutable or frozen) and by adding the generalized minors $\Delta _{{\varpi_j}, {\varpi_j}}$ modded out in $\CC[N_K]$ as frozen variables. The matrix $\widehat{B}$ of this lift is obtained as follows:
Extend the matrix $B$ of the construction of Goodearl and Yakimov ~\cite{GY} by $|J|$ rows labeled by the elements of $J$ such that the entries are
\[
  \widehat{b}_{jk} =
  \begin{cases}
\beta_j , & \text{if $\beta_j \neq 0$;} \\
- \alpha_j,& \text{else,}
  \end{cases}
  \]
where $\alpha_j$ and $\beta_j$ are as in Remark \ref{cluster mutation}.
\end{definition}

\begin{theorem}\label{main theorem}
Let $\left\{ (\textnormal{\textbf{x}},{B}) \right\}$ be the collection of seeds of the cluster algebra $\mathcal{A}_J$ of $\CC[N_K]$. The corresponding collection $\left\{ (\widehat{\textnormal{\textbf{x}}},\widehat{B}) \right\}$ constructed above forms a valid collection of seeds. In other words, if $(\textnormal{\textbf{x}},{B})$ and $(\textnormal{\textbf{x}}',{B'})$ are two seeds of the coordinate ring of the cell $\CC[N_K]$ such that $(\textnormal{\textbf{x}}',{B'}) = \mu_k(\textnormal{\textbf{x}},{B})$, then correspondingly $(\widehat{\textnormal{\textbf{x}}'},\widehat{B'})= \mu_k (\widehat{\textnormal{\textbf{x}}},\widehat{B})$.
\end{theorem}
\begin{proof}
Let $k$ be a mutable index in the construction of ~\cite{GY}. We need to show that $\mu_k(\widehat{B})=\widehat{B'}$. In other words, we need to show that the matrix entries of the mutation of $\widehat{B}$ match the ones coming from our construction (definition \ref{main definition}) using the mutation equations of the mutated seed $(\textnormal{\textbf{x}}',{B'})$. The equations are of the form
$$\widetilde{x'}_t \widetilde{x''}_t = \mu'_t \widetilde{M'}_t + \nu'_t \widetilde{L'}_t,$$
where $x'_t$ denotes the $t$th variable in the mutated extended cluster in a direction $k$ and $x''_t$ denotes the $t$th variable coming from a second mutation in a direction $t$. Let $\widehat{b'}_{st}$ denote the entry of position $s \times t$ in $\mu_k(\widehat{B})$. Obviously, if $s \notin J$ then $\widehat{b'}_{st}$ equals the $s \times t$ entry of $\mu_k(B)$, as the entries of $\widehat{B}$ and $B$ match when $s \notin J$. Consequently, the entries of the mutation of both of them coincide again when $s \notin J$. Assume now that $s \in J$ and $t=k$. Then by the fact that the construction of ~\cite{GY} is indeed a cluster algebra, we get that $M'_k=L_k$ and $L'_k=M_k$. This clearly makes $\alpha'_j=\beta_j$ and $\beta'_j=\alpha_j$. Since $\mu'_k$ and $\nu'_t$ are relatively prime, we see easily from the construction that the entry we get is $-\widehat{b}_{st}$ which equals $\widehat{b'}_{st}$ by the mutation formula.\\

Assume now that $t$ is a mutable index other than $k$. It suffices to show that in
$$\widetilde{x'}_t \widetilde{x''}_t = \mu'_t \widetilde{M'}_t + \nu'_t \widetilde{L'}_t,$$
the exponents of the minors of the monomials $\mu'_t$ and $\nu'_t$ match the formula of the matrix mutation. Equivalently, we may assume that $\mu'_t$ and $\nu'_t$ are as we want and then show that $\mu'_t \widetilde{M'}_t + \nu'_t \widetilde{L'}_t,$ is an element whose proj is $M'_t+L'_t$ and whose order is minimal with respect to $\preceq$. The first property is straightforward. Now,
\begin{equation} \label{lambda}
    \lambda \big( {M'_t} +{L'_t} \big)= \sum_{j \in J} a_j \big( {M'_t} +{L'_t} \big) \varpi_j= \sum_{j \in J} a_j \varpi_j
\end{equation}
where \\
$
\begin{aligned}
a_j &= a_j \big({M'_t} +{L'_t} \big)\\
&= \max \{s \mid (e^ \dagger _j )^s \big( {M'_t} +{L'_t} \big) \neq 0\}\\
& = \max \begin{Bmatrix} s \Biggm| (e^ \dagger _j )^s \begin{pmatrix}
\displaystyle \prod_{b_{it}+\frac{|b_{ik}| b_{kt} + b_{ik}|b_{kt}|}{2}>0} {{x'}}_i ^{b_{it}+\frac{|b_{ik}| b_{kt} + b_{ik}|b_{kt}|}{2}}\\
+ \\
\displaystyle \prod_{b_{it}+\frac{|b_{ik}| b_{kt} + b_{ik}|b_{kt}|}{2}<0} {{x'}}_i ^{- \Big(b_{it}+\frac{|b_{ik}| b_{kt} + b_{ik}|b_{kt}|}{2}\Big)} \\

\end{pmatrix} \neq 0 \end{Bmatrix}\\
%&=\max \Bigg\{s \text{ }| \text{ } (e^ \dagger _j )^s \bigg(\prod_{j \in J} \Delta_{\overline{\varpi_{j}},\overline{\varpi_{j}}}^{\alpha'_j} \prod_{b_{it}+\frac{|b_{ik}| b_{kt} + b_{ik}|b_{kt}|}{2}>0} \widetilde{{x'}}_i ^{b_{it}+\frac{|b_{ik}| b_{kt} + b_{ik}|b_{kt}|}{2}} + \prod_{j \in J} \Delta_{\overline{\varpi_{j}},\overline{\varpi_{j}}}^{\beta'_j}  \prod_{b_{it}+\frac{|b_{ik}| b_{kt} + b_{ik}|b_{kt}|}{2}<0} \widetilde{{x'}}_i ^{- \Big(b_{it}+\frac{|b_{ik}| b_{kt} + b_{ik}|b_{kt}|}{2}\Big)} \bigg)\neq 0 \Bigg\} \\
&= \max \begin{Bmatrix}
\displaystyle \sum_{b_{it}+\frac{|b_{ik}| b_{kt} + b_{ik}|b_{kt}|}{2}>0} a_j \Big({{x'}}_i ^{b_{it}+\frac{|b_{ik}| b_{kt} + b_{ik}|b_{kt}|}{2}} \Big),\\
\\
\displaystyle \sum_{b_{it}+\frac{|b_{ik}| b_{kt} + b_{ik}|b_{kt}|}{2}<0} a_j \Bigg( {{x'}}_i ^{- \Big(b_{it}+\frac{|b_{ik}| b_{kt} + b_{ik}|b_{kt}|}{2}\Big) } \Bigg)
\end{Bmatrix}.\\
\end{aligned}$\\

\noindent Note here that the last equality is obtained by the fact that $a_j(fg)=a_j(f) + a_j (g).$
%and $a_j(f+g)= \max \{a_j(f) + a_j(g) \}$ in this particular case.
Assume now, for some $j$, that the first sum is the maximum. Then, using the same fact once again, we clearly get that
$$a_j =
  \sum_i \Big( {b_{it}+\frac{|b_{ik}| b_{kt} + b_{ik}|b_{kt}|}{2}} \Big) a_j \big({{x'}}_i \big).$$
Recall also that $x'_i=x_i$ for $i \notin \{k,t\}$. So,
$$
a_j =
\sum_i \Big( {b_{it}+\frac{|b_{ik}| b_{kt} + b_{ik}|b_{kt}|}{2}} \Big) a_j \big({{x}}_i \big).$$
But by equation (\ref{lambda}), it follows that\\
\begin{align*}
    \lambda \big({M'_t} +{L'_t} \big) &= \sum_{j \in J} a_j \varpi_j\\
    &= \sum_{j \in J} \sum_{i}
    \begin{pmatrix}
    \displaystyle \bigg| {b_{it}+\frac{|b_{ik}| b_{kt} + b_{jk}|b_{kt}|}{2}} \bigg| a_j \big({{x}}_i \big)
    \end{pmatrix}
    \varpi_j.
\end{align*}
Now, since $\CC[G/P_{K}^{-}]$ is graded by the set of sums of the form $\sum_{i \in \NN} a_i \varpi_i$, the last equation implies that
\begin{align*}
L \left( \lambda \big({M'_t} +{L'_t} \big) \right)&=L \left( \sum_{j \in J} \sum_{i}
    \begin{pmatrix}
    \displaystyle \bigg| {b_{it}+\frac{|b_{ik}| b_{kt} + b_{jk}|b_{kt}|}{2}} \bigg| a_j \big({{x}}_i \big)
    \end{pmatrix}
    \varpi_j \right)\\
    & \supset \prod_{j \in J} \prod_{i} L \left( 
   \bigg| {b_{it}+\frac{|b_{ik}| b_{kt} + b_{ik}|b_{kt}|}{2}} \bigg| a_j \big({{x}}_i \big)
    \varpi_j \right)\\
    &= \prod_{j \in J} \prod_{i} L \left( 
   d_i a_j ({{x}}_i )
    \varpi_j \right),
    \end{align*}
    where
    $$d_i := \bigg| {b_{it}+\dfrac{|b_{ik}| b_{kt} + b_{ik}|b_{kt}|}{2}} \bigg|.$$
    So, we get that
    \begin{equation} \label{subset}
    L \left( \lambda \big({M'_t} +{L'_t} \big) \right) \supset
    \prod_{j \in J} \prod_{i} \underbrace{L \left( 
   a_j ({{x}}_i )
    \varpi_j \right) \cdots L \left( 
   a_j ({{x}}_i )
    \varpi_j \right)}_\text{$d_i$ times}.
    \end{equation}
    Note here that the sum over $i$ is the sum over the positive $d_i$'s only or the negative $d_i$'s only. A similar work with $L \left( \lambda \big({M_k} +{L_k} \big) \right)$ shows that
    \begin{align*}
L \left( \lambda \big({M_k} +{L_k} \big) \right)&=L \left( \sum_{j \in J} \sum_{i}
    \begin{pmatrix}
    \displaystyle |b_{ik}| a_j \big({{x}}_i \big)
    \end{pmatrix}
    \varpi_j \right)\\
    & \supset \prod_{j \in J} \prod_{i} L \left( 
   |b_{ik}| a_j \big({{x}}_i \big)
    \varpi_j \right).
    \end{align*}
    This implies that
    $$L \left( \lambda \big({M_k} +{L_k} \big) \right) \supset
    \prod_{j \in J} \prod_{i} \underbrace{L \left( 
    \varpi_j \right) \cdots L \left( 
    \varpi_j \right)}_\text{$|b_{ik}|a_j ({{x}}_i )$ times}.$$
    But since $\Delta_{\varpi_j,\varpi_j}$ is of degree $\varpi_j$, it follows that the possible occurrences of the exponents of $\Delta_{\varpi_j,\varpi_j}$ are the integers
    $$0,1,2,...,\sum_i a_j(x_i)|b_{ik}|.$$
    However, the minimality of
    $$a_j(\lambda(M_k+L_k))=\sum_{j \in J} \sum_i |b_i| a_j(x_i)$$
    shows that the only possible solution is $\sum_i a_j(x_i)|b_{ik}|$, because the rest are still available in some $L(\lambda)$'s in which $\lambda$ is less than $\lambda(M_k+L_k)$.
    Consequently, the only possibilities for $\alpha_j$ and $\beta_j$ are $0$ or $\sum_i a_j(x_i)|b_{ik}|.$ But, thank to the homogeneity of the construction of \cite{GY}, there exists a unique $j$ in which $a_j (x_i) \neq 0$. Hence, one of $\alpha_{jk}$ and $\beta_{jk}$ is $a_j (x_i) |b_{ik}|$ and the other is 0.\\
    Therefore, for every $i$ there is a unique $j$ in which one of the following must be true:
    $$a_j({x}_i)b_{it}=\pm \alpha_{jt}, \quad a_j({x}_i)b_{ik}=\pm \alpha_{jk} \quad \textnormal{and} \quad a_j({x}_i)|b_{ik}|=\alpha_{jk},$$
    or
    $$a_j({x}_i)b_{it}=\pm \beta_{jt}, \quad a_j({x}_i)b_{ik}=\pm \beta_{jk} \quad \textnormal{and} \quad a_j({x}_i)|b_{ik}|=\beta_{jk}.$$
    Now,
    \begin{align*}
    \left( \Delta_{\varpi_j,\varpi_j}^{a_j ({x}_i)} \widetilde{x}_i\right) ^{d_i} \in L \left( 
   d_i a_j ({{x}}_i )
    \varpi_j \right)\\
    \left( \Delta_{\varpi_j,\varpi_j}^{a_j ({x}_i)} \widetilde{x}_i\right) ^{{b_{it}+\frac{|b_{ik}| b_{kt} + b_{ik}|b_{kt}|}{2}}} \in L \left( 
   d_i a_j ({{x}}_i )
    \varpi_j \right)
    \end{align*}
    It is not hard now to see that $\left( \Delta_{\varpi_j,\varpi_j}^{a_j ({x}_i)} \widetilde{x}_i\right) ^{d_i}$ forms one factor of the monomial $\mu'_t \widetilde{
    M'}_t$ or the monomial $\nu'_t \widetilde{
    L'}_t$. The rest are similarly there. Since $L \left( \lambda \big({M'_t} +{L'_t} \big) \right)$ and $L \left( \lambda \big({M_k} +{L_k} \big) \right)$ are homogeneous ideals and since
    $$\widetilde{x_k x'_k}=\widetilde{x}_k \widetilde{x'}_k = \mu_k \widetilde{M}_k + \nu_k \widetilde{L}_k \in L \left( \lambda \big({M_k} +{L_k} \big) \right),$$
    it is again not difficult to combine these information to see that $$\widetilde{x'_t x''_t}=\widetilde{x'}_t \widetilde{x''}_t = \mu'_t \widetilde{M'}_t + \nu'_t \widetilde{L'}_t \in L \left( \lambda \big({M'_t} +{L'_t} \big) \right).$$
    This completes the proof.
\end{proof}

\begin{notation}
The cluster algebra contained in $\CC[G/{P_K^{-}}]$ and obtained from the preceding theorem will be denoted by $\widehat{\mathcal{A}_J}$ or simply $\widehat{\mathcal{A}}$ if the context is clear. 
\end{notation}

\begin{corollary}
Let $B$ be the matrix $\widetilde{B}^w$ of Theorem \ref{Schubert cell}. The pair
$$\bigg( \big\{ \widetilde{ D}_{\varpi_{i_k},w_{\leq k} \varpi_{i_k}}  \big\} \sqcup \{ \Delta_{\varpi_j, \varpi_j} \mid j\in J \}, \widehat{B} \bigg)$$
is an initial seed of the cluster algebra $\widehat{\mathcal{A}} \subset \CC[G/P_K^-]$.
\end{corollary}
\begin{proof}
Apply the construction of the previous theorem %together with the fact that
%$$\widetilde{D_{\varpi_{i_k},w_{\leq k} \varpi_{i_k}}}=\Delta_{\varpi_{i_k},w_{\leq k} \varpi_{i_k}}$$
to the initial seed $(D_{\varpi_{i_k},w_{\leq k} \varpi_{i_k}}, \widetilde{B}^w)$ of $\CC[N_K]$ (see Theorem \ref{Schubert cell}).  The mutable and frozen variables are described in Definition \ref{main definition}.
\end{proof}

\begin{remark}
One might think naively that the lift $\widetilde{ D}_{\varpi_{i_k},w_{\leq k} \varpi_{i_k}}$ is equal to a generalized minor. But this is not the case in general; see Example 10.3 of \cite{GLS}.
\end{remark}

\begin{remark}
In the simply-laced case, it is obvious that the construction of $\widehat{\mathcal{A}_J}$ matches the one of ~\cite{GLS}.
\end{remark}

\begin{remark}
By construction, it is clear that the extended clusters of $\mathcal{A}$ and $\widehat{\mathcal{A}}$ are in one-to-one correspondence. So, $\mathcal{A}$ and $\widehat{\mathcal{A}}$ must be of the same type (either both finite or both infinite).
\end{remark}

%\begin{notation}
%Let $\Sigma_J$ be the multiplicative submonoid of $\widehat{\mathcal{A}_J}$ generated by the minors $\Delta _{\varpi_j, \varpi_j}$ $(j\in J)$, where $\varpi_j$ is not minuscule.
%\end{notation}
%The following proposition will be implicitly used in the proof of the next theorem:
%\begin{proposition}
%The algebra $\widehat{\mathcal{A}_J}$ is a UFD.
%\end{proposition}

\begin{theorem}\label{equality theorem}
The localization of the homogeneous coordinate ring of the flag variety $\CC[G/P_K^-]$ by $\Delta_{\varpi_j, \varpi_j}$, $(j\in J )$ equals the localization of the cluster algebra $\widehat{\mathcal{A}}$ by the same elements. Namely,
$$\CC[G/P_K^-][\Delta_{\varpi_j, \varpi_j}^{-1}]_{j\in J}=\widehat{\mathcal{A}}[\Delta_{\varpi_j, \varpi_j}^{-1}]_{j\in J}.$$
%In particular, $\CC[G/P_K^-]$ is a cluster algebra whose initial seed is
%$$\bigg( \big\{ \widetilde{ D}_{\varpi_{i_k},w_{\leq k} \varpi_{i_k}}  \big\} \sqcup \{ \Delta_{\varpi_j, \varpi_j} \mid j\in J \}, \widehat{B} \bigg).$$
\end{theorem}
\begin{proof}
Throughout the proof, for any element in $\CC[G/P_{I\setminus \{j\}}^-]$ the term {degree} will be used to refer to the homogeneous degree of it.
Recall that $\CC [G/P_K^-]= \bigoplus_{\lambda \in \Pi_J} L(\lambda)$ is generated as a ring by the subspaces $L(\varpi_j) \subset \CC[G/ P_{I \setminus \{j\}}^-].$ Thus, it is generated by $\CC[G/ P_{I \setminus \{j\}}^-]$, where $j \in J$.
It is remarked in \cite{GLS} that if $J' \subset J$, then $\widehat{A_{J'}}$ is a subalgebra of $\widehat{A_{J}}$. Therefore, the result follows if localization of $\CC[G/ P_{I \setminus \{j\}}^-]$ by $\Delta_{\varpi_j,\varpi_j}$ is contained in the localization of $\widehat{\mathcal{A}_{\{j\}}}$, by the same element. We proceed by contradiction. Let $f \in \CC[G/ P_{I \setminus \{j\}}^-]$ such that $f \notin \widehat{\mathcal{A}_{\{j\}}}$ and its degree is minimal. Let $g=\textnormal{proj}(f) \in \CC[N_{I \setminus \{j\}}]$. Then $\textnormal{proj}(\widetilde{g}-f)=0$. Thus, we have that $\widetilde{g}-f$ belongs to the principal ideal $(\Delta_{\varpi_j,\varpi_j}-1)$, since
$$\CC[N_{I \setminus \{j\}}]=\bigslant{\CC[G/P_{I \setminus \{j\}}^{-}]}{(\Delta_{\varpi_j,\varpi_j}-1)}.$$
Consequently, there is some $h \in \CC[G/P_{I \setminus \{j\}}^{-}]$ such that
\begin{align*}
\widetilde{g}-f &=h\left(\Delta_{\varpi_j,\varpi_j}-1\right)\\
\widetilde{g}-f &=h\Delta_{\varpi_j,\varpi_j}-h.
\end{align*}
But note that the definition of $\widetilde{g}$ and the choice of $f$ imply that the degree of the whole left-hand side is less than or equal to degree of $f$. On the other hand, it is obvious that the degree of the right-hand side is the degree of $h$ plus $\varpi_j$. It follows that the degree of $h$ is less than the one of $f$. Therefore, by minimality, we get that $h \in \widehat{\mathcal{A}_{\{j\}}}$. Also, since $\Delta_{\varpi_j,\varpi_j} \in \widehat{\mathcal{A}_{\{j\}}}$, it follows that $h\Delta_{\varpi_j,\varpi_j} \in \widehat{\mathcal{A}_{\{j\}}}.$ Now, if the lifting $\widetilde{g} \in \widehat{\mathcal{A}_{\{j\}}},$ we have
\begin{equation}\label{f}
  f= 
    \underbrace{\widetilde{g}}_\text{$\in \widehat{\mathcal{A}_{\{j\}}}$}
    - 
    \underbrace{h\Delta_{\varpi_j,\varpi_j}}_\text{$\in \widehat{\mathcal{A}_{\{j\}}}$}
    +\underbrace{h}_\text{$\in \widehat{\mathcal{A}_{\{j\}}}$} \in \widehat{\mathcal{A}_{\{j\}}} \subset \widehat{\mathcal{A}_{\{j\}}}[\Delta_{\varpi_j, \varpi_j}^{-1}],
\end{equation}
which is a contradiction to $f$ being outside $\widehat{\mathcal{A}_{\{j\}}}$. Therefore, $\widetilde{g} \notin \widehat{\mathcal{A}_{\{j\}}}$. Now, write $g\in \CC[N_{I\setminus \{j\}}]$ as $g=\sum_{i=1}^r c_im_i,$ where each $m_i$ is a product of cluster variables (might not be from the same seed) and each $c_i$ is a scalar. This forces the $m_i$'s to be distinct.
%Denote by $d_i$ the degree of each $\widetilde{m_i}$. Let $d=\max\{d_1,...,d_r\}$.
%We have that
%$$\sum_{i=1}^r c_i \Delta_{\varpi_j,\varpi_j}^{d-d  _i}\widetilde{m_i}$$
%is an element of $\widehat{\mathcal{A}_{\{j\}}}$ whose projection is $g$.
Recall that $\CC[N_K]$ can be identified with
$$\CC[N_K]=\left\{\dfrac{f}{\prod_{j \in J} \Delta_{\varpi_j,\varpi_j}^{a_j}} \mid f\in L \bigg(\sum_{j\in J} a_j \varpi_j \bigg) \right\}.$$
By the uniqueness and minimality of the tilde map in Lemma \ref{raise}, this can be refined to
$$\CC[N_K]=\left\{\dfrac{f}{\prod_{j \in J} \Delta_{\varpi_j,\varpi_j}^{a_j}} \mid f\in L \bigg(\sum_{j\in J} a_{j} \varpi_j \bigg) \textnormal{ and the } a_j \textnormal{'s are minimal} \right\}.$$
As $J=\{j\}$, we can use the second identification to write each $m_i$ as $\dfrac{f_i}{\Delta_{\varpi_j,\varpi_j}^{a_{i,j}}}$, where $f_i \in L(a_{i,j}\varpi_j)$ and $a_{i,j}$ is minimal with this property. Clearly, the degree $d_i$ of $\widetilde{m_i}$ is $a_{i,j} \varpi_j$. It is not hard to see that $\widetilde{m_i}=f_i$ for all $i=1,...,r$.
%
%In the construction of Goodearl and Yakimov of the cluster structure of $\CC[N_{I\setminus\{j\}}]$, denote the set of mutable variable indices of the initial seed by $R$ and the one of frozen variable indices by $S$. By Laurentness, each ${m_i}$ can be expressed as
%\begin{equation}
%{m_i}= \dfrac{p_i}{\prod_{r \in R} {D}_{\varpi_r,w\varpi_r}^{a_{i,r}}},
%\end{equation}
%where
%$$p_i \in \CC[{D}_{\varpi_r,w\varpi_r},{D}_{\varpi_s,w\varpi_s}]_{r\in %R,\textnormal{ } s\in S}.$$
%On the other hand, as a consequence of Lemma \ref{raise}, $m_i$ can be expressed uniquely as
%\begin{equation}
%{m_i}= \dfrac{f}{\prod_{j \in J} {\Delta}_{\varpi_j,w\varpi_j}^{a_{j}}},
%\end{equation}
%where $f \in L(\sum_{j \in J} a_j \varpi_j)$
%From the construction of Goodearl and Yakimov of $\CC[N_I\setminus\{j\}]$, denote the set of mutable variable indices of the initial seed by $R$ and the one of frozen variable indices by $S$. By Laurentness, each $\widetilde{m_i}$ can be expressed as
%\begin{equation}
%\widetilde{m_i}= \dfrac{p_i(\widetilde{D}_{\varpi_r,w\varpi_r},\widetilde{D}_{\varpi_s,w\varpi_s},\Delta_{\varpi_j,\varpi_j})}{q_i(\widetilde{D}_{\varpi_r,w\varpi_r})},
%\end{equation}
%where
%$$p_i \in \CC[\widetilde{D}_{\varpi_r,w\varpi_r},\widetilde{D}_{\varpi_s,w\varpi_s},\Delta_{\varpi_j,\varpi_j}]_{r\in R,\textnormal{ } s\in S, \textnormal{ }j \in J} \quad \text{and} \quad q_i \in \CC[\widetilde{D}_{\varpi_r,w\varpi_r}]_{r \in R}$$
%where $p_i$ and $q_i$ are polynomials $r$ runs through the set of mutable variable indices of the initial seed
%
As distinct elements lift by the tilde map to distinct elements by Lemma \ref{raise}, we get that the $\widetilde{m_i}$'s are distinct. Consequently, the $f_i$'s are distinct. Let $a_j:=\max{\{a_{i,j} \mid i=1,...,r\}}$.
Now, if
\begin{equation}\label{min eq}
\sum_{i=1}^r \left( c_i \Delta_{\varpi_j,\varpi_j}^{a_j-a_{i,j}}f_i \right)
\end{equation}
is a lift of $g=\sum_{i=1}^r c_im_i$ in the minimal way, then we get that $\widetilde{g} \in \widehat{\mathcal{A}_{\{j\}}}$, which is a contradiction. Otherwise, there is an element of lower degree in which $\widetilde{g}$ is equal to that element. Note that the multiplication of $\widetilde{g}$ by $\Delta_{\varpi_j,\varpi_j}^{s_j}$ for some positive integer $s_j$ gives an element whose degree is equal to the degree of the element in \ref{min eq} and whose projection is equal to $g$. Since the projection of each homogeneous piece $L(\lambda)$ to $\CC[N_K]$ is injective, we get that
$$\widetilde{g}=\dfrac{\sum_{i=1}^r \left( c_i \Delta_{\varpi_j,\varpi_j}^{a_j-a_{i,j}}f_i \right)}{\Delta_{\varpi_j,\varpi_j}^{s_j}}.$$
Clearly, by (\ref{f}) this implies again that $f \in \widehat{\mathcal{A}_{\{j\}}}[\Delta_{\varpi_j, \varpi_j}^{-1}]$. Consequently, the result follows, that is,
$$\CC[G/P_K^-][\Delta_{\varpi_j, \varpi_j}^{-1}]_{j\in J}=\widehat{\mathcal{A}}[\Delta_{\varpi_j, \varpi_j}^{-1}]_{j\in J}.$$
\end{proof}

\begin{conjecture}
The homogeneous coordinate ring of the flag variety $\CC[G/P_K^-]$ equals the cluster algebra $\widehat{\mathcal{A}}$. In particular, $\CC[G/P_K^-]$ is a cluster algebra whose initial seed is
$$\bigg( \big\{ \widetilde{ D}_{\varpi_{i_k},w_{\leq k} \varpi_{i_k}}  \big\} \sqcup \{ \Delta_{\varpi_j, \varpi_j} \mid j\in J \}, \widehat{B} \bigg).$$
\end{conjecture}

\begin{remark}
Using the proof of the previous theorem, the conjecture is equivalent to proving that if $g=\sum_{i=1}^r c_i m_i$ is written where the number of terms is minimal, then
$$\widetilde{g}=\sum_{i=1}^r \left( c_i \Delta_{\varpi_j,\varpi_j}^{a_j-a_{i,j}}\widetilde{m_i} \right).$$
\end{remark}

\begin{example}
Let $G$ be a semisimple algebraic group of type $B_3$, say $G={SO}_{2(3)+1}={SO}_7$, $J=\{3\}$ and $K=I \setminus J=\{1,2\}.$ Consider the longest word
%$$w_0=s_1s_2s_1s_3s_2s_1s_3s_2s_3.$$
\[w_0={s_1s_2s_1} \boldsymbol{s_3s_2s_1s_3s_2s_3}.\]
The subword $w=s_3s_2s_1s_3s_2s_3$ generates $N_K$. Since $s(3)=s(5)=s(6)=\infty$ and $s(k) \neq \infty$ for $k\in \{1,2,4 \}$, we get that the mutable variables are indexed by $1,2,4$ and the frozen ones are indexed by $3,5,6$ using the function $s$ of Theorem \ref{Schubert cell}. Therefore, by the same theorem, the exchange matrix of the cluster algebra structure of $\CC[N_K]$ is
%$$(\widetilde{B}^w)_{jk} = \begin{cases}
%1, & \text{if } j=p(k), \\
%-1, & \text{if } j=s(k),\\
%a_{i_j i_k}, & \text{if } j<k<s(j)<s(k),\\
%-a_{i_j i_k}, & \text{if } k<j<s(k)<s(j),\\
%0, & \text{otherwise.}
%\end{cases}$$
\[
\begin{blockarray}{cccc}
1 & 2 & 4 \\
\begin{block}{(ccc)c}
  0 & a_{i_1 i_2} & 1 & 1 \\
  -a_{i_2 i_1} & 0 & a_{i_2 i_4} & 2 \\
  -1 & -a_{i_4 i_2} & 0 & 4 \\
  -a_{i_3 i_1} & -a_{i_3 i_2} & 0 & 3 \\
  0 & -1 & -a_{i_5 i_4} & 5 \\
  0 & 0 & -1 & 6\\
\end{block}
\end{blockarray}
 \]
 \[=
\begin{blockarray}{cccc}
1 & 2 & 4 \\
\begin{block}{(ccc)c}
  0 & -2 & 1 & 1 \\
  1 & 0 & -1 & 2 \\
  -1 & 2 & 0 & 4 \\
  0 & 1 & 0 & 3 \\
  0 & -1 & 1 & 5 \\
  0 & 0 & -1 & 6\\
\end{block}
\end{blockarray}
 \]
where the column labels denote the cluster variables and the row labels denote the extended cluster variables, as usual. Also, the extended cluster variables $D_{\varpi_{i_j},w_{\leq j} \varpi_{i_j}}$ are
\begin{align*}
%j=1
j=1 &\implies D_{\varpi_{3},s_3 \varpi_{3}}; & \textnormal{(mutable)}\\
%j=2
j=2 &\implies D_{\varpi_{2},s_3s_2 \varpi_{2}}; & \textnormal{(mutable)}\\
%j=3
j=3 &\implies D_{\varpi_{1},s_3s_2s_1 \varpi_{1}}; & \textnormal{(frozen)}\\
%j=4
j=4 &\implies D_{\varpi_{3},s_3s_2s_1s_3 \varpi_{3}}; & \textnormal{(mutable)}\\
%j=5
j=5 &\implies D_{\varpi_{2},s_3s_2s_1s_3s_2 \varpi_{2}}; & \textnormal{(frozen)}\\
%j=6
j=6 &\implies D_{\varpi_{3},s_3s_2s_1s_3s_2s_3 \varpi_{3}}.& \textnormal{(frozen)}
\end{align*}
Therefore, by Theorem $\ref{main theorem}$, the following list of variables form an initial extended cluster of $\widehat{\mathcal{A}} \subset \CC[G/P_K^-]$
\begin{align*}
&\widetilde{D}_{\varpi_{3},s_3 \varpi_{3}}; &\textnormal{(mutable)}\\
&\widetilde{D}_{\varpi_{2},s_3s_2 \varpi_{2}}; &\textnormal{(mutable)}\\
&\widetilde{D}_{\varpi_{3},s_3s_2s_1s_3 \varpi_{3}}; &\textnormal{(mutable)}\\
&\widetilde{D}_{\varpi_{1},s_3s_2s_1 \varpi_{1}}; &\textnormal{(frozen)}\\ %\Delta_{\varpi_3,\varpi_3},
&\widetilde{D}_{\varpi_{2},s_3s_2s_1s_3s_2 \varpi_{2}}; &\textnormal{(frozen)}\\
&\widetilde{D}_{\varpi_{3},s_3s_2s_1s_3s_2s_3 \varpi_{3}}; &\textnormal{(frozen)}\\
&\Delta_{\varpi_3,\varpi_3}. &\textnormal{(frozen)}
\end{align*}
Consequently, the extended exchange matrix $\widehat{B}$ of $\widehat{\mathcal{A}}$ attached to this extended cluster is
% \[
%\widehat{B}=\begin{blockarray}{cccc}
% & 4 & 5 & 6 \\
%\begin{block}{c(ccc)}
%  4 & 0 & -a_{i_4 i_5} & 1 \\
%  5 & a_{i_5 i_4} & 0 & 0 \\
%  6 & -1 & 0 & 0 \\
%  1 & 1 & a_{i_1 i_5} & a_{i_1 i_6} \\
%  2 & a_{i_2 i_4} & 1 & a_{i_2 i_6} \\
%  3 & 0 & 0 & 0\\
%  \cmidrule(lr){2-4}
%  j \in J & ? & ? & 0 \\
%\end{block}
%\end{blockarray}
%\textnormal{ } .\]
  \[
\widehat{B}=\begin{blockarray}{cccc}
 1 & 2 & 4 \\
\begin{block}{(ccc)c}
  0 & -2 & 1 & 1 \\
  1 & 0 & -1 & 2 \\
  -1 & 2 & 0 & 4 \\
  0 & 1 & 0 & 3 \\
  0 & -1 & 1 & 5 \\
  0 & 0 & -1 & 6\\
  \cmidrule(lr){1-3}
  -1 & 0 & 0 & j \in J\\
\end{block}
\end{blockarray}
.\]
\end{example}

\section*{Acknowledgements}
The author thanks Jimmy Dillies, Bernard Leclerc, Bach Nguyen and Milen Yakimov for useful discussions in preparing this paper. We are grateful to the anonymous referee for pointing out an error in the first version of the paper.

\end{document}